\documentclass{amsart}
\usepackage{fullpage}
\usepackage{fullwidth}
\usepackage{amsmath,amssymb}
\usepackage{graphicx}
\usepackage{mathtools}
\usepackage{hyperref}

\begin{document}

\newcommand{\e}{\text{e}}
\newcommand{\Z}{\mathbb{Z}}
\newtheorem{theorem}{Theorem}[section]
\newtheorem{lemma}[theorem]{Lemma}
\newtheorem{corollary}[theorem]{Corollary}
\newtheorem{question}[theorem]{Question}
\newtheorem{conjecture}[theorem]{Conjecture}
\newtheorem{claim}{}[theorem]
\theoremstyle{remark}
\newtheorem*{remark*}{Remark}
\newcommand{\ignore}[1]{}
\newcommand{\todo}[1]{TODO: {\em #1}}

\newcommand{\BB}{\mathcal{B}}
\newcommand{\ZZ}{\mathcal{Z}}
\newcommand{\NS}{\mathcal{N}}
\newcommand{\N}{\mathbb{N}}
\newcommand{\cl}{\text{cl}}

\newcommand{\MM}{\mathbb{M}}
\newcommand{\SP}{\mathbb{S}}
\newcommand{\HH}[1]{\mathcal{H}\!\left(#1\right)}

\newcommand{\statetheorem}[2]{\expandafter\def\csname thmcontent@#1\endcsname {#2}\begin{theorem}\label{#1}#2\end{theorem}\newtheorem*{thm:#1}{Theorem \ref{#1}}}
\newcommand{\repeattheorem}[1]{\begin{thm:#1}\csname thmcontent@#1\endcsname \end{thm:#1}}

\newcommand{\noproof}{\mbox{ }\hfill$\square$}

\mathtoolsset{centercolon}

\title{On the number of bases of almost all matroids}

\begin{abstract} For a matroid $M$ of rank $r$ on $n$ elements, let  $b(M)$ denote the fraction of bases of $M$ among the subsets of the ground set with cardinality $r$.
We show that  $$\Omega(1/n)\leq 1-b(M)\leq O(\log(n)^3/n)\text{ as }n\rightarrow \infty$$ for asymptotically almost all matroids $M$ on $n$ elements. We derive that asymptotically almost all matroids on $n$ elements  
(1) have a $U_{k,2k}$-minor, whenever $k\leq O(\log(n))$,
(2) have girth $\geq \Omega(\log(n))$, 
(3) have Tutte connectivity $\geq \Omega(\sqrt{\log(n)})$, and 
(4) do not arise as the truncation of another matroid.

Our argument is based on a refined method for writing compressed descriptions of any given matroid, which allows bounding the number of matroids in a class relative to  the number of sparse paving matroids. 
\end{abstract}
\author{Rudi Pendavingh}
\author{Jorn van der Pol}
\address{Eindhoven University of Technology, Eindhoven, the Netherlands}
\email{R.A.Pendavingh@tue.nl, ~J.G.v.d.Pol@tue.nl}
\thanks{This research was supported by the Netherlands Organisation for Scientific Research (NWO) grant 613.001.211.}

\maketitle

\section{Introduction}
\subsection{Matroid asymptotics} This paper is concerned with the properties that are satisfied by most matroids as the size of their ground set tends to infinity. Precisely, for a matroid property $P$ we consider the asymptotic fraction
$$\frac{ |\{M\in \MM_n: M\text{ has property }P\}|}{|\MM_n|}$$
where $\MM_n$ denotes the set of matroids with ground set $E=\{1,\ldots, n\}$. If this fraction tends to 1 as $n\rightarrow \infty$, then we say that $P$ holds for (asymptotically) almost all matroids. 

Perhaps the first to make any comment on this issue were Crapo and Rota in the early 1970's, when they speculated  that ``paving matroids may predominate in any asymptotic enumeration of matroids'' in \cite{CrapoRotaBook}. 
In 2011, Mayhew, Newman, Welsh, and Whittle note that  ``even the most elementary questions about the properties of almost all matroids are currently unanswered''. They break the spell by proving a bound on the asymptotic proportion of connected matroids, and list several attractive and natural conjectures on the properties that asymptotically hold for almost all matroids in \cite{MayhewNewmanWelshWhittle2011}. 
Cloteaux proves in \cite{Cloteaux2010} that if $\epsilon>0$, then $b(M)\geq 1/n^{5/2+\epsilon}$ for almost all matroids on $n$ elements, where $b(M)$ denotes the proportion of bases among the subsets of the ground set of rank $r(M)$. 
Oxley, Semple, Warshauer,  and Welsh show in \cite{OxleySempleWarshauerWelsh2013} that asymptotically almost all matroids are 3-connected.  The present authors show in \cite{PvdP2015} that if $N$ is $U_{2,k}, U_{3,6}$, or one of several other matroids on 6 elements, then almost all matroids have $N$ as a minor, and in \cite{PvdP2015B} that almost all matroids on $n$ elements have rank between $n/2-\beta\sqrt{n}$ and $n/2+\beta\sqrt{n}$ whenever $\beta>\sqrt{\ln(2)/2}$. 

\subsection{A conjecture} A matroid $M$ of rank $r$ is a {\em sparse paving matroid} if and only if each dependent set of cardinality $r$ is a circuit-hyperplane of $M$. The following conjecture is a version of the  original statement of Crapo and Rota, and equivalent to a conjecture from \cite{MayhewNewmanWelshWhittle2011}.
\begin{conjecture}\label{conj:sp} As $n\rightarrow \infty$, asymptotically almost all matroids on $n$ elements are sparse paving.
\end{conjecture}
If this conjecture were true, then several other asymptotic properties of matroids would follow with little extra work, as it would suffice to establish the property for almost all sparse paving matroids. 
For example, almost all sparse paving matroids contain a fixed uniform matroid, are highly connected, have high girth, and the analogous statements for general matroids would follow. We next describe our intuition supporting Conjecture~\ref{conj:sp}, which we partially substantiated to obtain the results of this paper.

Consider the {\em Johnson graph} $J(E,r)$, whose vertices are the subsets of $E$ of cardinality $r$, and in which two vertices $X,Y$ are adjacent exactly if $|X\cap Y|=r-1$.
The Johnson graph will serve as an `ambient space' for all the matroids on ground set $E$ and of rank $r$. In what follows, we will write $G:=J(E,r)$. Given any matroid $M$ on $E$ of rank $r$, the set of bases $\BB$ of $M$ consists of subsets of $E$ of cardinality $r$, so that $\BB\subseteq V(G)$. Let $K:=V(G)\setminus\BB$ be the dependent $r$-sets or `non-bases' of $M$, and consider the induced subgraph $G[K]$. 
We state a straightforward consequence of the matroid basis-exchange axiom.
\begin{lemma} Let $G:=J(E,r)$, let $K\subseteq V(G)$, and let $C_1,\ldots, C_k$ be the components of $G[K]$. Then the following are equivalent:
\begin{enumerate}
\item $V(G)\setminus K$ is the set of bases of a matroid on $E$,
\item for each $i$, $V(G)\setminus C_i$ is the set of bases of a matroid on $E$.\noproof
\end{enumerate}
\end{lemma}
This lemma implies that given our matroid $M$, we can create other matroids by removing components of $G[K]$, or introducing new components $C$ as long as such $C$ are the sets of non-bases of a matroid on $E$. If $C=\{X\}$ is a singleton, then $X$ is a circuit-hyperplane of $M$, and the removal of such $C$ from the non-bases is known as {\em relaxing} a circuit-hyperplane. The reverse operation, introducing a singleton as a component, will also always yield another matroid, as any singleton $C$ is the set of non-bases of a matroid. Intuitively, bigger sets $C$ will be harder to introduce as a component: they are more likely to fuse with existing components, and perhaps more significantly, not just any non-singleton $C$ is the set of non-bases of a matroid on $E$.

In this paper we crucially distinguish between the set of circuit-hyperplanes $W\subseteq K$ and the remaining set of dependent sets $U=K\setminus W$, the non-bases in `big components'. Being a set of singleton components, $W$ is a stable set of $G$. Replacing $W$ with any other stable set $W'$ of $G$ which is disjoint from $U\cup N(U)$ will produce the non-bases $K'=U\cup W'$ of another matroid. It will be intuitively clear that this construction may yield a significant number of distinct matroids, the more so if $U$ does not cover too much of $G$. 

If we subdivide all the matroids of a given rank $r$ on $E$ into cohorts according to the cardinality of $U$, the number of matroids in each cohort will be roughly the number of ways to cover $|U|$ vertices of $G$ with big components, times the number of ways to position stable sets $W$ away from such $U$. As the cardinality of $U$ increases, the number of compatible $W$ decreases. In view of the inertia of big components compared to singletons, we do not expect that the increasing numbers of possible $U$ will compensate this decrease.   This suggests that the cohorts where $|U|$ is small will be more populated. Conjecture \ref{conj:sp} takes the extreme, but not unlikely position that eventually as $|E|\rightarrow \infty$, almost all matroids will be in the cohort where $|U|=0$.

\subsection{Our results} For any matroid $M$, let $W(M)$ denote the set of circuit-hyperplanes of $M$ and let $U(M)$ denote the remaining non-bases of $M$. The central result of this paper is the following.\protect\footnote{Throughout the paper, we will use~$\log$ to denote the base-2 logarithm, and~$\ln$ to denote the natural logarithm.}
\begin{figure}
\includegraphics[width=.5\textwidth]{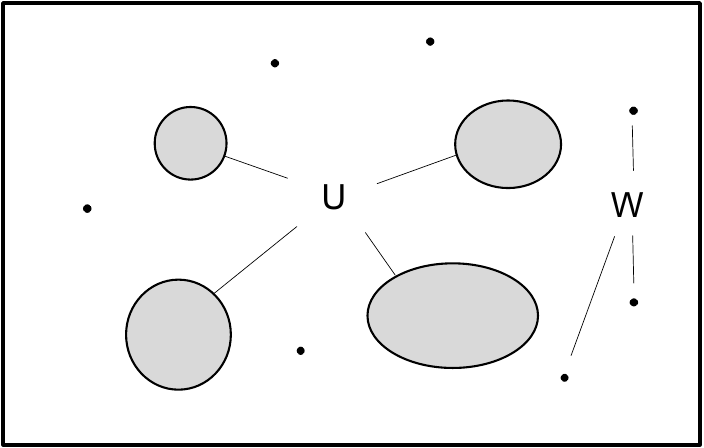}
\caption{\label{fig:mat} The non-bases $K$ of a matroid induce a subgraph of the Johnson graph. }
\end{figure}
\begin{theorem} \label{thm:ns}There is a  $c>0$ and a class of matroids $\NS\subseteq \MM$ containing almost all matroids, such that 
$$\log \left|\{U(M): M\in  \MM_{n,r}\cap \NS\}\right|\leq c \frac{\log(n)^3}{n^2}\binom{n}{n/2}$$
for all $0\leq r\leq n$.
\end{theorem}
If Conjecture \ref{conj:sp} is true, then Theorem \ref{thm:ns}  follows:  taking $\NS$ to be the class of sparse paving matroids, we would obtain $\{U(M): M\in  \MM_{n,r}\cap \NS\}=\{\emptyset\}$. The bound of Theorem \ref{thm:ns} is very weak by comparison, but the  number is still vanishing compared to the best known lower bound on the number of matroids on $n$ elements, $\log m(n)\geq \binom{n}{n/2}/n$. Since each matroid $M$ is determined by the pair $(U(M), W(M))$, we obtain the following generic tool for showing that a class of matroids is thin, i.e.\ that almost all matroids are outside the class.
\begin{corollary}\label{cor:thin}
	There is a~$c > 0$ such that if~$\mathcal{M} \subseteq \MM$ is a class of matroids satisfying
	$$
		\log |\{W(M) : M \in \mathcal{M} \cap \MM_{n,r}\}| \le \left(1-c\frac{\log(n)^3}{n}\right) \log m(n),
	$$
	for all $0 \le r \le n$ and $n$ sufficiently large, then~$\mathcal{M}$ is thin.
\end{corollary}
All our asymptotic results will essentially rely on this corollary, and we develop two methods to prove the bound on the variation in $W(M)$ within a class.

The first method  is derived from Shearer's Entropy Lemma, and proceeds from the assumption that $U(M)$ is relatively large.
Using this result, we will be able to show that $d(M)$, the fraction of dependent sets of cardinality $r(M)$,  is $O(\log(n)^3/n)$ for asymptotically almost all  matroids. 
Counting dependent sets which necessarily exist in a matroid without a given uniform matroid minor, we also show that almost all matroids on $n$ elements have a $U_{k,2k}$-minor whenever $k\leq O(\log(n))$.
Counting the number of dependent sets  in a matroid with a small circuit, or a small-order separation, we establish that asymptotically almost all matroids on $n$ elements have girth $\geq \Omega(\log(n))$ and connectivity $\geq \Omega(\sqrt{\log(n)})$. The latter proves Conjecture 1.5 of \cite{MayhewNewmanWelshWhittle2011}, which states that for any fixed $k$, almost all matroids are $k$-connected.

The second, more straightforward way to bound the variation in $W(M)$ is by assuming a uniform upper bound on the cardinality of each $W(M)$. Using an extension of a theorem from \cite{BPvdP2015}, which bounds the number of stable sets in a regular graph, we are able to show that $d(M)\geq \Omega(1/n)$ for almost all matroids. This will imply that most matroids do not arise by truncation. 


Theorem \ref{thm:ns} is derived from a method for writing a compressed description of matroids. This method was first used in \cite{BPvdP2015} to prove an absolute upper bound on the number $m(n)$ of matroids on $n$ elements, and further developed in \cite{PvdP2015B} to yield an upper bound on $m(n)$ relative to the number of sparse paving matroids $s(n)$. We have now refined this procedure to yield, given any matroid $M$, either an encoding of $M$ which also very concisely describes $U(M)$, or very many distinct encodings of $M$. Since the overall number of encodings is bounded close to the number of sparse paving matroids, the matroids which have many encodings will be few by comparison. The remaining matroids, for which $U(M)$ can be described compactly, will constitute the class $\NS$ in Theorem \ref{thm:ns}.

\subsection{Structure of the paper} In Section \ref{sec:pre}, we give preliminaries on matroids, the Johnson graph, and review the work on counting stable sets and matroids from \cite{BPvdP2015} and \cite{PvdP2015B}. We make slight extensions and modifications to these results, which will be of use later in the paper.  Section \ref{sec:concise} is devoted to the proof of Theorem \ref{thm:ns}. In Section \ref{sec:few}, we show that in asymptotically almost all matroids, only an $O(\log(n)^3/n)$ fraction of the $r$-sets may be dependent. In Section \ref{sec:many}, we show that in asymptotically almost all matroids, at least a $\Omega(1/n)$ fraction of the $r$-sets must be dependent. In Section \ref{sec:discussion}, we present conjectures and discuss openings for future work.


\section{\label{sec:pre}Preliminaries}
\subsection{Matroids} We refer to Oxley's book \cite{OxleyBook} as a general reference on matroids. We adhere to the notation used in that book, with a few exceptions. 

Let $[n]:=\{1,\ldots,n\}$. We write
$$\MM_{n,r}:=\{ M\text{ a matroid}: E(M)=[n], r(M)=r\}, ~\MM_n:=\bigcup_r\MM_{n,r},  \text{ and }\MM:=\bigcup_n \MM_n.$$
For the cardinalities of these classes, we denote $m(n,r):=|\MM_{n,r}|$ and $m(n):=|\MM_n|$.

A matroid $M$ of rank $r$ is {\em paving} if all its circuits have cardinality at least $r$, and $M$ is {\em sparse paving} if both $M$ and $M^*$ are paving. We write 
$$s(n,r):=\left|\{ M\in \MM_{n,r}: M\text{ is sparse paving}\}\right|, \text{ and }s(n):=\left|\{ M\in \MM_{n}: M\text{ is sparse paving}\}\right|.$$

Given a set of matroids $\mathcal{M}\subseteq \MM$, we say that {\em (asymptotically) almost all matroids are in} $\mathcal{M}$ if 
$$\frac{|\mathcal{M}\cap \MM_n|}{|\MM_n|}\rightarrow 1\text{ as } n\rightarrow \infty$$
and that $\mathcal M$ is {\em thin} if the same limit is 0. 

We will apply the following result, due to Mayhew, Newman, Welsh, and Whittle \cite{MayhewNewmanWelshWhittle2011}.
\begin{theorem}\label{thm:simple} Almost all matroids are loopless and coloopless.
\end{theorem}

\subsection{The Johnson graph} If $E$ is a finite set and $r\in \N$, then we write 
$$\binom{E}{r} : = \{ X \subseteq E : |X| = r \}.$$
The Johnson graph $J(E,r)$ is defined as the graph with vertex set $\binom{E}{r}$, in which two vertices $X,Y\in \binom{E}{r}$ are adjacent if and only if $|X\cap Y|=r-1$. 
We abbreviate $J(n,r) := J([n],r)$. If $X\in \binom{E}{r}$, then 
$$N(X):=\{X-e+f: e\in X, f\in E\setminus X\}$$
will denote the neighborhood of $X$ in the Johnson graph.

The Johnson graphs are relevant to matroid theory by the following result.
\begin{lemma} Let $S\subseteq \binom{E}{r}$. Then $S$ is a stable set of $J(E,r)$ if and only if  $\binom{E}{r}\setminus S$ is the set of bases of a sparse paving matroid on $E$.\noproof\end{lemma}
It follows that $s(n,r)$ equals the number of stable sets of the Johnson graph $J(n,r)$.
Using this lemma, Knuth \cite{Knuth1974} proved a lower bound on the number of sparse paving matroids based on a construction for stable sets in the Johnson graph, which was subsequently improved by Graham and Sloane \cite{GrahamSloane1980}:
\begin{theorem} \label{thm:knuth} For any $0<r<n$, we have $s(n,r)\geq 2^{\binom{n}{r}/n}.$\noproof\end{theorem}
Hence $s(n)\geq s(n,\lfloor n/2\rfloor)\geq  2^{\binom{n}{n/2}/n}$, a fact which we will repeatedly use in this paper for showing that a class of matroids on $n$ elements is small compared to $m(n)$, using that $s(n)\leq m(n)$.

\subsection{Binomial coefficients}
We use the following standard bound on sums of binomial coefficients (see \cite[Prop.~1.4]{Jukna2001}):
$$\sum_{i=0}^k\binom{n}{i}\leq \left(\frac{\e n}{k}\right)^k.$$ 
We also apply the following well-known bounds on  the central binomial coefficient, that are easily proved using Stirling's approximation,
$$\sqrt{\frac{2}{\pi}}\frac{2^n}{\sqrt{n}}\left(1-\frac{1}{8n}\right)\leq \binom{n}{n/2} \leq \sqrt{\frac{2}{\pi}}\frac{2^n}{\sqrt{n}}\text{ for all $n\geq 1$}.$$
 The shorthand $\binom{n}{n/2}:=\binom{n}{\lfloor n/2\rfloor}$ is used throughout this paper.

\subsection{Counting stable sets in regular graphs}
Let $i(G):=|\{S\subseteq V(G): S\text{ a stable set of }G\}|$. We describe a method for bounding $i(G)$ for regular graphs $G$, which was adapted in \cite{BPvdP2015} from a method due to Kleitman and Winston \cite{KleitmanWinston1982}. For further applications of this method see the survey of Samotij \cite{Samotij2014}. 

If $G=(V,E)$ is a graph and $A\subseteq V$, we write $G[A]$ for the subgraph of $G$ induced on $A$ and $\Delta G$ for the maximum degree of a vertex of $G$. The {\em eigenvalues} of $G$ are the eigenvalues of the adjacency matrix of $G$.

In the statement of Lemma~\ref{lem:KW1} below, and throughout the remainder of the paper, we will use the phrase ``$X$ is uniquely determined by~$Y$'' in the following precise way. Whenever the phrase is used, there will be a free variable (describing the context) as well as a bound variable. We say that~$X$ is uniquely determined by~$Y$ if there is a function~$f$ such that~$X = f(Y)$. The function~$f$ is allowed to depend on the free variable, but not on the bound variable.

\begin{lemma}\label{lem:KW1}
	Let $G=(V,E)$ be regular of degree $d\geq 1$, with smallest eigenvalue $-\lambda$, and let~$V$ be linearly ordered.
	For each $K\subseteq V$, there exist sets $S, A \subseteq V$ such that
	\begin{enumerate}
		\item $|S| \le \left\lceil\sigma|V|\right\rceil$, $|A| \le \alpha |V|$;
		\item $S \subseteq K \subseteq S \cup N(S) \cup A$;
		\item $A$ is uniquely determined by~$S$.
	\end{enumerate}
	Here, $\sigma := \ln(d+1)/(d+\lambda)$ and $\alpha := \lambda/(d+\lambda)$.
\end{lemma}

\begin{remark*}
	The linear order on the vertex set~$V$ in the statement of Lemma~\ref{lem:KW1} allows us to select the {\em canonical max-degree vertex} among the vertices in any subset $X\subseteq V$. This is the vertex~$v$ of maximum degree in~$G[X]$; if there are multiple vertices of maximum degree, then we take~$v$ to be the smallest with respect to the order on~$V$.
	
	The proof of Lemma~\ref{lem:KW1} implicitly constructs a function~$f$ to show that~$A$ is uniquely determined by~$S$. The canonical choice of a vertex of maximum degree in~$G[X]$ is essential in this construction. Therefore, the function~$f$ depends on both~$G$ and the linear order on its vertices (these are the context, or free variable, in the terms specified in the remark preceeding the statement of the lemma), but not on~$K$ (the bound variable).
\end{remark*}

The proof of Lemma~\ref{lem:KW1} is the content of~\cite[Section~4]{BPvdP2015}. We will sketch the proof, and refer to the cited paper for details.

\begin{proof}[Proof sketch.]
	We construct~$S$ and~$A$, starting with~$S=\emptyset$ and~$A=V$. While~$|A| > \alpha V$, pick the canonical max-degree vertex~$v$ in~$A$. If~$v \not\in K$, replace~$A \leftarrow A-v$. Otherwise, replace~$S \leftarrow S+v$ and $A \leftarrow A \setminus \left(\{v\} \cup N(v)\right)$. Repeat.
	It was established in \cite{BPvdP2015} that by the time  $|A|\leq \alpha |V|$ is attained, we have $|S|\leq \left\lceil\sigma|V|\right\rceil$.
	
	By construction, it is clear that $S\subseteq K\subseteq S\cup N(S)\cup A$. Since the steps of the procedure can be retraced using $S$ in place of~$K$, the set $A$ is uniquely determined by $S$.
\end{proof}

In~\cite{PvdP2015B}, we showed that we can obtain~$\Delta G[A] < \lambda$ at the expense of making~$S$ slightly larger. The proof in~\cite{PvdP2015B} was given for the Johnson graph, but it holds for general regular graphs as well.

\begin{lemma}\label{lem:KW2} Let $G=(V,E)$ be regular of degree $d\geq 1$, with smallest eigenvalue $-\lambda$, and let~$V$ be linearly ordered.
For each $K\subseteq V$, there exist sets $S, A\subseteq V$ such that 
\begin{enumerate}
\item $|S|\leq \sigma^+ |V|$, $|A|\leq \alpha |V|$, and $\Delta G[A]<\lambda$;
\item $S\subseteq K\subseteq S\cup N(S)\cup A$; and
\item $A$ is uniquely determined by $S$.
\end{enumerate}
Here $\sigma^+:=(\ln(d+1)+1)/(d+\lambda)$ and $\alpha:=\lambda/(d+\lambda)$.
\end{lemma}
\begin{proof}
	Let $S', A' \subseteq V$ be the sets obtained from an application of Lemma~\ref{lem:KW1}. Note that~$|S'| \le \sigma |V|$, and $|A'| \le \alpha |V|$, where~$\sigma = \ln(d+1)/(d+\lambda)$, and $\alpha= \lambda/(d+\lambda)$. We construct~$S$ and~$A$ starting with~$S = S'$ and~$A = A'$.

	While $\Delta G[A] \ge \lambda$, pick the canonical max-degree vertex~$v \in A$. If~$v \not\in K$, replace $A\leftarrow A-v$. Otherwise, replace~$S\leftarrow S+v$ and~$A\leftarrow A\setminus\left(\{v\} \cup N(v)\right)$. Repeat.
	
	Each time a vertex is added to~$S$, at least~$\lambda$ vertices are removed from~$A$, so~$|S\setminus S'| \le |A'|/\lambda \le (\alpha/\lambda)|V|$. It follows that~$|S| \le \sigma|V| + (\alpha/\lambda)|V| = \sigma^+|V|$.
	By construction, $\Delta G[A] < \lambda$, and $S \subseteq K \subseteq S \cup N(S) \cup A$.
	
	Since~$A'$ is uniquely determined by~$S'$, and the procedure described in this proof can be retraced using~$S \setminus S'$ instead of~$K$, it follows that~$A$ is uniquely determined by~$S$.
\end{proof}

The following theorem is an application of Lemma~\ref{lem:KW2}. In it, we drop the superscript~$+$ from~$\sigma^+$.

\begin{theorem}\label{thm:stable}Let $G=(V,E)$ be regular of degree $d\geq 1$, with smallest eigenvalue $-\lambda$. Then 
 $$  i(G)\leq \sum_{i=0}^{\left\lceil \sigma N\right\rceil} \binom{N}{i} 2^{\alpha N}, $$
where $\sigma:=(\ln(d+1) + 1)/(d+\lambda)$, $\alpha:=\lambda/(d+\lambda)$, and $N:=|V|$.
\end{theorem}
\proof Fix a linear order on~$V$. Let $K\subseteq V$ be a stable set of $G$. By Lemma \ref{lem:KW2}, there are sets $S, A\subseteq V$ with $|S|\leq \sigma N, |A|\leq \alpha N$, such that $S$ determines $A$ and $S\subseteq K\subseteq S\cup N(S)\cup A$. Since $K$ is a stable set $K\cap N(S)=\emptyset$, and hence $K=S\cup (K\cap A)$. 
There are $\binom{N}{i}$ different sets~$S$ of cardinality~$i$, and for each fixed~$S$, there are at most~$2^{|A|}$ possible~$K\cap A$. The theorem now follows from the bounds on~$|S|$ and~$|A|$.
\endproof

As noted above, we have $s(n,r)=i(J(n,r))$, so Theorem \ref{thm:stable} may be used to derive upper bounds on the number of sparse paving matroids.
Note that the vertices of $J(n,r)$ are linearly ordered by the lexicographic order.
The Johnson graph $J(n,r)$ has degree $d=r(n-r)$ and smallest eigenvalue $-\lambda=-\min\{r, n-r\}$ (see e.g.\ \cite[Theorem 9.1.2]{BrouwerCohenNeumaier1989}). Thus the two constants that arise as Lemma \ref{lem:KW2} is applied to $J(n,r)$ are
$$\sigma_{n,r}:=\frac{\ln(r(n-r)+1)+1}{r(n-r)+\min\{r,n-r\}}\text{ and }\alpha_{n,r}:=\frac{\min\{r,n-r\}}{r(n-r)+\min\{r,n-r\}}.$$

\begin{lemma}\label{lem:sigma}
	For all $0 < r < n$, $\left\lceil\sigma_{n,r}\binom{n}{r}\right\rceil\leq 9\ln(n)\binom{n}{n/2}/n^2$ and $\alpha_{n,r}\binom{n}{r}\leq 2\binom{n}{n/2}/n$.
\end{lemma}

\begin{proof}
	Define
	$$
		f(n,r) := \frac{2\ln n}{r(n-r)} \binom{n}{r}.
	$$
	Note that $f(n,n-r) = f(n,r)$, and $\sigma_{n,r}\binom{n}{r} + 1 \le f(n,r)$. A calculation reveals that~$f(n,r-1) \le f(n,r)$ whenever $1 < r \le \lfloor n/2\rfloor$. Hence
	$$
		\left\lceil \sigma_{n,r} \binom{n}{r}\right\rceil \le f(n,r) \le f(n, \lfloor n/2\rfloor) \le \frac{9 \ln(n)}{n^2} \binom{n}{n/2}
	$$
	whenever $0 < r < n$, as required. The bound~$\alpha_{n,r} \le 2/n$ is trivial.
\end{proof}

Using these bounds in Theorem \ref{thm:stable}, and using~$s(n,0) = s(n,n) = 1$, we obtain the following result from \cite{BPvdP2015}.
\begin{theorem}$\log s(n)\leq \frac{2+o(1)}{n}\binom{n}{n/2}$ as $n\rightarrow \infty$.\noproof
\end{theorem}

\subsection{Counting matroids}\label{ss:counting-matroids} We review the method for counting matroids from \cite{BPvdP2015}, which was an extension of the Kleitman-Winston method for counting stable sets, and which proceeds by writing a compressed description of matroids. Precisely, for a matroid $M=(E,\BB)\in \MM_{n,r}$, we describe the set of non-bases $K:=\binom{E}{r}\setminus \BB$, viewed as a set of vertices of the Johnson graph $J(n,r)$.

For our matroid description, the set $S$ from Lemma \ref{lem:KW2} is complemented by a so-called {\em partial flat cover} $\ZZ$.
Suppose $M$ is a matroid on ground set $E$, and let $X\subseteq E$. If $F$ is a flat of $M$ so that $|X\cap F|>r(F)$, then $X$ is dependent, as it contains the dependent set $X\cap F$. In that case, we say that the pair $(F, r(F))$ {\em covers} $X$. A single flat $F$ combined with its rank is a relatively compact way to certify the dependency of many sets $X$ in the matroid. We will be interested in describing the dependent $r$-sets in the neighborhood of a dependent $r$-set $X$ of $M$, hence the following definitions. 
\begin{enumerate}
\item If  $r_M(X)<|X|-1$, then we put $\ZZ(M,X):=\{(\cl_M(X), r_M(X))\}$. 
\item If  $r_M(X)=|X|-1$, then we put $\ZZ(M,X):=\{(\cl_M(C), r_M(C)), (\cl_M(X), r_M(X))\}$ where $C$ is the unique circuit contained in $X$.
\end{enumerate}
The following lemma is elementary.
\begin{lemma} Suppose $M$ is a matroid on ground set $E$, of rank $r$. If  $X\in \binom{E}{r}$ is dependent, then each dependent set $Y\in N(X)$ is covered by some $(F,r)\in \ZZ(M,X)$. \noproof
\end{lemma}
Here $N(X)=\{X-e+f: e\in X, f\in E\setminus X\}$, the neighborhood of $X$ as a vertex of the Johnson graph.
\begin{lemma} \label{lem:BPvdP}
Let $0 < r < n$, and write $G:=J(n,r)$. For all~$M \in \MM_{n,r}$, there is a stable set $S$ of $G$, a collection $\ZZ\subseteq 2^E\times \{0,\ldots, n-1\}$ and a set $A\subseteq V(G)$ such that
\begin{enumerate}
\item $|S|\leq \sigma_{n,r} \binom{n}{r}$, $|\ZZ|\leq 2|S|$, $|A|\leq \alpha_{n,r} \binom{n}{r}$, and $\Delta G[A]<\min\{r,n-r\}$;
\item $S \subseteq V(G) \setminus \BB(M) \le S \cup N(S) \cup A$; and
\item $K\setminus A$ is uniquely determined by the pair $(S,\ZZ)$, and  $A$ is uniquely determined by $S$. 
\end{enumerate}
\end{lemma}
\proof Apply Lemma \ref{lem:KW2} to the Johnson graph $G:=J(n,r)$ and the subset $K\subseteq V(G)$ to obtain sets $S$, $A$. Put 
$$\ZZ:=\bigcup_{X\in S}\ZZ(M,X).$$
Then $|\ZZ|\leq 2|S|$ as each $\ZZ(M,X)$ has at most two elements. If $X\in S$, then each dependent $Y\in N(X)$ is covered by some element of $\ZZ$, hence $\ZZ$ determines $K\cap N(S)$. As $S\subseteq K\setminus A\subseteq S\cup N(S)$, it follows that $K\setminus A$ is uniquely determined by the pair $(S,\ZZ)$, as required.
\endproof
In what follows, it will be convenient to have a bound on the number of possible pairs $(S,\ZZ)$ which may arise from the application of Lemma \ref{lem:BPvdP}. So let 
$$z(n,r)= \left|\left\{(S,\ZZ): ~S\subseteq \binom{[n]}{r}, ~|S|\leq \left\lceil\sigma_{n,r}\binom{n}{r}\right\rceil, ~\ZZ\subseteq 2^{[n]}\times\{0,\ldots, n-1\},~  |\ZZ|\leq 2|S|\right\}\right|.$$
Put $\zeta(n):=57\frac{\log(n)^2}{n^2}\binom{n}{n/2}$.
\begin{lemma}\label{lem:zbound} $\log z(n,r)\leq  \zeta(n)$ for sufficiently large $n$.\end{lemma}
\proof Counting the number of subsets $S$ of size at most $\left\lceil\sigma_{n,r}\binom{n}{r}\right\rceil$ from a set of size at most $\binom{n}{n/2}$ times the number of subsets $\ZZ$ of size at most $2|S|$ from a set of size $n2^n$ we get
$$z(n,r) \leq \left(\sum_{i=0}^{N}\binom{\binom{n}{n/2}}{i}\right)\left(\sum_{i=0}^{2N}\binom{n2^n}{i}\right)\leq \left(\frac{\e \binom{n}{n/2}}{N}\right)^{N}\left(\frac{\e n2^n}{2N}\right)^{2N}$$
for any $N\geq \left\lceil\sigma_{n,r}\binom{n}{r}\right\rceil$. Clearly the upper bound increases as a function of $N$ while $N\leq \binom{n}{n/2}/2$.
By Lemma~\ref{lem:sigma}, we can take
$$\max_r\left\lceil\sigma_{n,r}\binom{n}{r}\right\rceil\leq  \frac{9\ln(n)}{n^2}\binom{n}{n/2}=:N.$$
It follows that $\e \binom{n}{n/2}/N\leq n^2$ and $\e n2^n/(2N)\leq n^{7/2}$, provided that~$n$ is sufficiently large, and hence
$$\log z(n,r)\leq N\log (n^2) + 2N\log (n^{7/2})\leq 81 \frac{\ln(n)\log(n)}{n^2}\binom{n}{n/2}\leq \zeta(n),$$
 as required.\endproof
We obtain the main result of \cite{BPvdP2015}.
\begin{theorem}$\log m(n)\leq \frac{2+o(1)}{n}\binom{n}{n/2}$ as $n\rightarrow \infty$.
\end{theorem}
\proof Let~$0 < r < n$ be given.
By Lemma \ref{lem:BPvdP}, for each~$M = (E,\BB) \in \MM_{n,r}$ there is a pair $(S,\ZZ)$ which uniquely determines a set $A\subseteq \binom{E}{r}$ of cardinality at most $\alpha_{n,r}\binom{n}{r}$, and which also determines the set $K\setminus A$, where $K=\binom{E}{r} \setminus \BB$.
Hence $M$ is determined by the pair $(S,\ZZ)$ and the subset $K\cap A$ of a set $A$ that is fixed by $S$.

The number of possible pairs $(S,\ZZ)$ was bounded in Lemma \ref{lem:zbound}, and the number of possible subsets $K\cap A$ of a fixed $A$ is at most $2^{|A|}$. Hence 
$$\log |\MM_{n,r}|\leq \zeta(n)+\alpha_{n,r}\binom{n}{r}\leq O\left(\frac{\log(n)^2}{n^2}\binom{n}{n/2}\right)+ \frac{2}{n}\binom{n}{n/2}\leq \frac{2+o(1)}{n}\binom{n}{n/2}\text{ as } n\rightarrow \infty,$$
and $\log m(n)=\log |\MM_n|\leq \log (n\max_r |\MM_{n,r}|)=\log(n) + \max_r\log |\MM_{n,r}|= \frac{2+o(1)}{n}\binom{n}{n/2}$, as required.\endproof


\section{\label{sec:concise}A concise description of matroids}

\subsection{A new compression algorithm} We refine the compression algorithm from \cite{BPvdP2015} and \cite{PvdP2015B}, to obtain either very compact matroid encodings, or very many encodings of the same matroid. 

We make use of the following elementary matroid lemmas, which we will apply to derive the dependency of certain subsets  in a matroid, based on information about the dependency of other subsets.
\begin{lemma} \label{lem:induce1}Let $M=(E,\BB)\in\MM_{n,r}$ be without loops or coloops, let $K:=\binom{E}{r}\setminus\BB$, and let $X\in \binom{E}{r}$.  If \begin{enumerate}
\item $\{X-e+f: f\in E\setminus X\}\subseteq K$ for some $e\in X$, or 
\item $\{X-e+f: e\in X\}\subseteq K$ for some $f\in E\setminus X$, 
\end{enumerate}
then $X\in K$.\noproof
\end{lemma}
\begin{lemma} \label{lem:induce2}Let $M=(E,\BB)\in\MM_{n,r}$, let $K:=\binom{E}{r}\setminus\BB$, and let $X\in \binom{E}{r}$. Suppose $r_M(X)=r-1$. Then there is a unique circuit $C\subseteq X$ and a unique cocircuit $D\subseteq E\setminus X$, and 
$$K\cap N(X)=\{X-e+f: e\in X\setminus C\text{ or } f\in E\setminus X\setminus D\}.$$ \noproof
\end{lemma}
We will now prove the main technical lemma of this paper, which is an extension of Lemma \ref{lem:BPvdP}. The proof uses ideas developed in \cite{PvdP2015B}, in particular the use of the above lemmas to sparsify the encoding. 

\begin{lemma} \label{thm:encoding}Let~$0 < r < n$.
For all $M=(E,\BB)\in\MM_{n,r}$, there exist sets $S, W$, a set $\ZZ\subseteq 2^E \times \{0, 1, \ldots, n-1\}$, a number $t\in\N$,
and a collection~$\mathcal{T} \subseteq 2^{\binom{E}{r}}$ of cardinality at least~$2^t$, such that
\begin{enumerate}
\item $|S|\leq \left\lceil\sigma_{n,r}\binom{n}{r}\right\rceil$, and $|\ZZ|\leq 2|S|$,
\end{enumerate}
while for all~$T \in \mathcal{T}$
\begin{enumerate}
\setcounter{enumi}{1}
\item $|T|=t$, $T \cap W = \emptyset$, and $T\cup W$ is a stable set of $J(n,r)$;
\item $M$ is uniquely determined by $(S, \ZZ, T\cup W)$; and
\item $U(M)$ is uniquely determined by $(S, \ZZ, T)$.
\end{enumerate}
\end{lemma}
The proof of Lemma~\ref{thm:encoding} introduces a large number of different sets. Figure~\ref{fig:enc} may be useful in keeping track of the relations between these sets.
\proof We apply Lemma \ref{lem:BPvdP} to obtain a stable set $S$ in the Johnson graph $G=J(n,r)$, a set $A\subseteq \binom{E}{r}$ and a collection $\ZZ\subseteq 2^E\times [n]$, so that $A$ and $K\setminus A$ are uniquely determined by $(S,\ZZ)$. From the lemma, we have:
\begin{claim} \label{clA}$|S|\leq \left\lceil\sigma_{n,r}\binom{n}{r}\right\rceil$ and  $|\ZZ|\leq 2|S|$.\noproof
\end{claim}
We proceed to encode $K\cap A$,  exploiting the additional information from Lemma \ref{lem:BPvdP} that $\Delta(G[A])<\min\{r,n-r\}$, and our assumption that $M$ does not contain loops or coloops.

For a set $X\in\binom{E}{r}$, we abreviate
$$X_e:=\{X-e+y: y\in E\setminus X\}\text{ and }X^f:=\{X-x+f: x\in X\}.$$
Using Lemma \ref{lem:induce1} we can use our knowledge of $K\setminus A$ to derive that further elements of $A$ are necessarily contained in $K$. So let $P$ be the smallest subset of elements from $A$ whose dependency in the matroid is implied by a repeated application of that lemma. That is, let $P\subseteq A$ be the smallest set such that 
\begin{enumerate}
\item if $X\in A$ and there is an $e\in X$ such that $X_e\subseteq P\cup (K\setminus A)$, then   $X\in P$; and
\item if $X\in A$ and there is an $f\in E\setminus X$ such that $X^f\subseteq P\cup (K\setminus A)$, then  $X\in P$.
\end{enumerate}
The following is clear.
\begin{claim} The pair $(S,\ZZ)$ uniquely determines $P$, and $P\subseteq K\cap A$.\noproof\end{claim}
We set $A':=A\setminus P$, then $K=(K\setminus A)\cup P\cup (K\cap A')$. We have reduced our task to encoding  $K\cap A'$.
\begin{claim} Let $X\in A'$. Then there is an $e\in X$ such that $X_e\cap A'=\emptyset$ and an $f\in E\setminus X$ such that  $X^f\cap A'=\emptyset$.\end{claim}
\proof As $A'\subseteq A$, we have $\Delta G[A']\leq \Delta G[A]< \min\{r,n-r\}$. So $X$ has at most $\min\{r, n-r\}-1$ neighbors in $A'$. Hence $X$ cannot have a neighbor in $X_e\cap A'$ for each $e\in X$, and $X$ cannot have a neighbor in $X^f\cap A'$ for each $f\in E\setminus X$. \endproof
For each $X\in A'$, let $f^* \equiv f^*(X)$ be the minimal $f\in E\setminus X$ such that $X^f\cap A'=\emptyset$, and put 
$$C(X):=\{x\in X: X-x+f^*\in X^{f^*}\setminus (K\setminus A')\}.$$ 
Similarly, let $e^* \equiv e^*(X)$ be the minimal $e\in X$ such that $X_e\cap A'=\emptyset$ and put
$$D(X):=\{y\in E\setminus X: X-e^*+y\in X_{e^*}\setminus (K\setminus A')\}.$$
\begin{claim} The pair $(S,\ZZ)$ uniquely determines $C(X)$ and $D(X)$ for each $X\in A'$.\end{claim}
\proof $(S,\ZZ)$ determines $A'$ and $K\setminus A'$, and these in turn determine $C(X)$ and $D(X)$.\endproof
\begin{claim} Let $X, Y\in A'$, such that $Y=X-e+f$. If $X\in K$, and  $e\not\in C(X)$ or $f\not\in D(X)$, then $Y\in K$.
\end{claim} 
\proof Suppose $X\in K\cap A'$. By definition, $C(X)=\{x\in X: X-x+f^*\in X^{f^*}\setminus (K\setminus A')\}$, where $X^{f^*}\cap A'=\emptyset$, so that in fact 
$$C(X)=\{x\in X: X-x+f^*\in X^{f^*}\setminus K\}.$$
If $C(X)=\emptyset$, then $X^{f^*}\subseteq K$, which would contradict that $X\not\in P$. Hence there exists a $g\in C(X)$, such that $X-e+g\not\in K$ and hence $r_M(X)\geq r_M(X-e+g)-1=r-1$. We conclude tht $r_M(X)=r-1$, and Lemma \ref{lem:induce2} applies to $X$. Since we have shown $C(X)\neq\emptyset$ and by a dual argument $D(X)\neq \emptyset$, it follows that the unique circuit contained in $X$ is $C(X)$ and the unique cocircuit disjoint from $X$ is $D(X)$, and hence $$K\cap N(X)=\{X-x+y: x\in X\setminus C(X)\text{ or } y\in (E\setminus X)\setminus D(X)\}.$$ 
The claim follows.\endproof
\begin{figure}
\includegraphics[width=.5\textwidth]{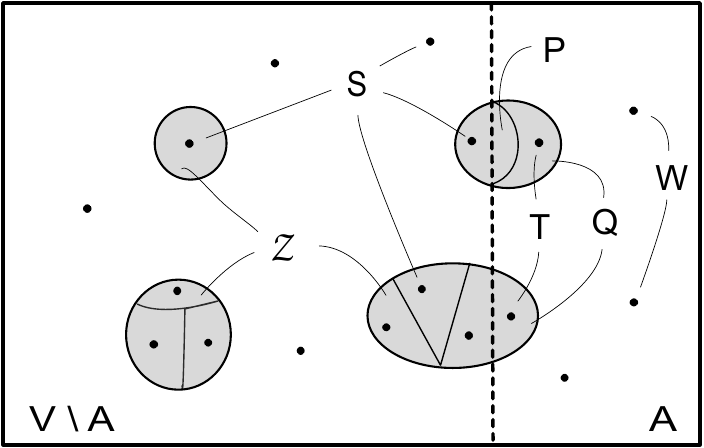}
\caption{\label{fig:enc} Encoding the non-bases $K$ of a matroid. }
\end{figure}
For each $X\in A'$, let $Q(X)$ be the set of vertices from $A'$ whose dependency would follow from $X\in K$ by an iterated application of the last claim. That is, let $Q(X)$ be the smallest subset $Q$ of $A'$ so that 
\begin{enumerate}
\item $X\in Q$ and 
\item if $X'\in Q$, $Y\in A'$ are such that $Y=X'-e+f$, and  $e\not\in C(X')$ or $f\not\in D(X')$, then $Y\in Q$.
\end{enumerate}
\begin{claim} For each $X\in A'$, the pair $(S,\ZZ)$ uniquely determines $Q(X)$.\end{claim}
\proof $Q(X)$ is determined only by the values of $C(Y), D(Y)$ for $Y\in A'$, which in turn are determined by $(S,\ZZ)$.\endproof
\begin{claim} Let $C$ be a component of $G[K\cap A']$, and let $X\in C$. Then $C=Q(X)$.
\end{claim}
\proof Clearly $C\subseteq Q(X)$, since $Q(X)$ will contain all of $N(Y)\cap K\cap A'$ whenever $Y\in Q(X)$. But on the other hand, the closure operation (2) defining $Q$ will never introduce elements from $N(Y)\setminus K$ for $Y\in Q(X)$, so $C\supseteq Q(X)$ as well. \endproof
We distinguish between the singleton components of $G[K\cap A']$ and the other components. Put
$$W:=\{X\in A': \{X\}\text{ is a component of }G[K\cap A']\},$$
and let $C_1,\dots, C_t$ be the~$t$ components of $G[K\cap A']$ of cardinality $>1$. Let~$\mathcal{T}$ be the collection of all~$T\subseteq \bigcup_{i=1}^t C_i$ such that $|C_i \cap T| = 1$ for all $i = 1, \ldots, t$. The following is clear.
\begin{claim} \label{clB}
$|\mathcal{T}| \ge 2^t$.
\end{claim}
For the remainder of the proof, let~$T \in \mathcal{T}$.
\begin{claim}\label{clC} $|T|=t$, and $T\cup W$ is a stable set of $G$.\end{claim}
\proof That $|T|=t$ is clear, and by construction $T\cup W$ contains a single vertex from each component of an induced subgraph of $G$, hence must be stable.\endproof
\begin{claim} \label{clD}$M$ is uniquely determined by  $(S, \ZZ, T\cup W)$.
\end{claim}
\proof $K=(K\setminus A) \cup P\cup \bigcup_{X\in T\cup W} Q(X)$, and $(K\setminus A)$, $P$ and $Q$ are determined by $(S,\ZZ)$.\endproof
\begin{claim} $W\subseteq W(M)$.
\end{claim}
\proof Suppose $X\in W$. If $X$ is not a circuit-hyperplane of $M$, then there is an $e\in X$ so that $X_e\subseteq K$ or an $f\in E\setminus X$ so that $X^f\subseteq K$. In the former case, since $X$ is an isolated vertex of $G[K\cap A']$, we have $X_e\cap A'=\emptyset$, and hence $X_e\subseteq P\cup (K\setminus A)$. By definition of $P$, it follows that $X\in P$, a contradiction. The analogous argument settles the case when $X^f\subseteq K$.\endproof
\begin{claim}\label{clE}$U(M)$ is uniquely determined by  $(S, \ZZ, T)$.
\end{claim}
\proof $W\subseteq W(M)$, so $U(M)\subseteq K\setminus W = (K\setminus A) \cup P\cup\bigcup_{X\in T} Q(X) =:\tilde{U}$, which  is determined by $(S,\ZZ, T)$. Now $U(M)$ arises from $\tilde{U}$ by removing the isolated vertices of $G[\tilde{U}]$ from $\tilde{U}$.
\endproof
The lemma now follows from claims \ref{clA}, \ref{clB}, \ref{clC}, \ref{clD}, and \ref{clE}.\endproof
For the remainder of this section, let $t(M)$ denote the value of $t$ for which the above lemma holds true.

In \cite{PvdP2015B} we used the precursor of Lemma~\ref{thm:encoding}  to show that $\log m(n)=(1+o(1))\log s(n)$ as $n\rightarrow\infty$. We also derived the following result, which we will apply in Section~\ref{sec:few}.
\begin{theorem} \label{thm:rank_bound}If $\beta>\sqrt{\ln(2)/2}$, then almost all matroids $M$ on $n$ elements have $|r(M)-n/2|<\beta\sqrt{n}$.\noproof
\end{theorem}

\subsection{A proxy for the class of sparse paving matroids} 
In this section, we prove Theorem~\ref{thm:ns}. The key insight is the following. For a given matroid~$M$, $t(M)$ is either ``small'' or ``large''. In the former case, Lemma~\ref{thm:encoding} shows that~$U(M)$ has a concise description; in the latter case, $M$ has many different encodings.

We define the class~$\NS$, which appears in Theorem~\ref{thm:ns}, essentially as the class of matroids for which~$t(M)$ is small.
For each $n,r\in \N$, let  $$\NS_{n,r}:=\left\{M\in\MM_{n,r}: M\text{ has no loops or coloops}, t(M)\leq 2\zeta(n)\right\},$$ and put $\NS=\bigcup \NS_{n,r}$.

\begin{lemma} \label{lem:ns_all} Asymptotically almost all matroids are contained in $\NS$.\end{lemma}
\proof We consider the set $\mathcal{M}=\bigcup \mathcal{M}_{n,r}$, where
$$\mathcal{M}_{n,r}:=\left\{M\in\MM_{n,r}: M\text{ has no loops or coloops}, t(M)> 2\zeta(n)\right\}.$$
By Theorem \ref{thm:simple}, almost all matroids have neither loops nor coloops. 
To show that almost all matroids are in $\NS$, it will therefore suffice to show that $\mathcal M$ is thin. 

We bound $|\mathcal{M}_{n,r}|$. By Lemma~\ref{thm:encoding} each matroid $M\in \mathcal{M}_{n,r}$  has at least $2^{2\zeta(n)}$ encodings $(S, \ZZ, T\cup W)$, where $S$ and $T\cup W$ are stable sets of $J(n,r)$,  $|S|\leq \sigma_{n,r}\binom{n}{r}$ and $\ZZ\subseteq 2^{[n]}\times \{0,\ldots, n-1\}$ is such that  $|\ZZ|\leq 2|S|$.  Recall that~$z(n,r)$ counts the number of pairs~$(S,\ZZ)$, hence
$$|\mathcal{M}_{n,r}|\leq  z(n,r)s(n,r)2^{-2\zeta(n)}.$$
By Lemma~\ref{lem:zbound}, $\log z(n,r) \le \zeta(n)$. Note that $\zeta(n) \ge \log n$ for sufficiently large~$n$. It follows that
$$
	\frac{|\MM_n \cap \mathcal{M}|}{m(n)} \le \frac{\sum_{r=0}^n |\mathcal{M}_{n,r}|}{s(n)} \le \frac{s(n)}{n s(n)} = \frac{1}{n} \to 0
$$
as~$n\to\infty$.
\endproof

Let $\mathcal{U}_{n,r}:=\{U(M): M\in  \NS_{n,r}\}$, and let $\Upsilon(n):=5 \log(n)\zeta(n)\leq O\left(\log(n)^3/n^2\binom{n}{n/2}\right)$.
\begin{lemma} \label{lem:ns_ubound}For all sufficiently large $n$, and all $r$ such that $0\leq r\leq n$, we have $\log |\mathcal{U}_{n,r}|\leq \Upsilon(n).$
\end{lemma}
\proof Note that~$\mathcal{U}_{n,0} = \mathcal{U}_{n,n} = \emptyset$, which proves the lemma for~$r=0$ and~$r=n$. So we may assume that $0<r<n$.
Let $N:=\binom{n}{n/2}$. For each of the matroids $M\in  \NS_{n,r}$, there is a triple $(S,\ZZ, T)$ which determines $U(M)$, where $|S|\leq \sigma_{n,r}N$, $|\ZZ|\leq 2|S|$, and $|T|\leq 2\zeta(n)$.  There are at most $z(n,r)$ possibilities for the pair $(S,\ZZ)$  and at most
$$ \tau(n,r):=\sum_{i=0}^{2\zeta(n)}\binom{N}{i}\leq \left(\frac{\e N}{2\zeta(n)}\right)^{2\zeta(n)}$$
possibilities for $T$, given these upper bounds on the cardinalities of $S,T,\ZZ$. By Lemma \ref{lem:zbound} 
$$\log \left|\mathcal{U}_{n,r}\right|\leq \zeta(n)+  \log\tau(n,r)\leq \zeta(n) + 2\zeta(n)\log\left(\frac{\e N}{2\zeta(n)}\right). $$
Hence
$\log \left|\mathcal{U}_{n,r}\right|\leq 2\zeta(n)\log(n^2)(1+o(1))\leq \Upsilon(n)$ 
for sufficiently large $n$, as required.
\endproof
The typical use of the class $\NS$ is for showing that a class of matroids is thin, which generic argument  goes as follows.

\begin{lemma}\label{lem:thin}
	There exists a constant~$c > 0$ such that if~$\mathcal{M} \subseteq \MM$ is a class of matroids satisfying
	$$
		\log |\{W(M) : M \in \mathcal{M} \cap \MM_{n,r}\}| \le \left(1-c\frac{\log(n)^3}{n}\right) \log m(n),
	$$
	for all $0 \le r \le n$ and $n$ sufficiently large, then~$\mathcal{M}$ is thin.
\end{lemma}

\begin{proof}
	Since asymptotically almost all matroids are in $\NS$, it suffices to show that  $\NS\cap \mathcal{M}$ is thin. Since each $M\in  \MM_{n,r}$ is determined by the pair $(U(M), W(M))$, it follows that 
$$|\MM_{n,r}\cap \NS\cap \mathcal{M}|\leq |\{U(M): M\in \NS_{n,r}\}|\cdot |\{W(M): M\in \mathcal{M}\cap \MM_{n,r}\}|,$$
and hence
$$\log |\MM_{n,r}\cap \NS\cap \mathcal{M}| \leq \Upsilon(n) + \left(1-c\frac{\log(n)^3}{n}\right) \log m(n).$$
From Theorem~\ref{thm:knuth}, we know that~$\Upsilon(n) = O\left(\frac{\log(n)^3}{n} \log m(n)\right)$, so~$|\MM_{n,r} \cap \NS\cap \mathcal{M}| \leq m(n)/(n+1)^2$, provided~$c$ is sufficiently large. We thus obtain
$$
\frac{|\MM_{n}\cap \NS\cap \mathcal{M}|}{m(n)}=
\sum_{r=0}^{n} \frac{|\MM_{n,r}\cap \NS\cap \mathcal{M}|}{m(n)}\leq 
(n+1)\max_r\frac{|\MM_{n,r}\cap \NS\cap \mathcal{M}|}{m(n)}\leq \frac{1}{n+1}\rightarrow 0$$
as~$n\to\infty$, as required.
\end{proof}

Theorem \ref{thm:ns} of the introduction follows directly from Lemma \ref{lem:ns_all} and Lemma \ref{lem:ns_ubound}, and its Corollary \ref{cor:thin} is Lemma \ref{lem:thin}. 


\section{\label{sec:few}Most matroids have few non-bases}

\subsection{Stable sets in vertex-transitive graphs} 
The following is the Product Theorem from \cite{ChungGrahamFranklShearer1986}, which derives from  {Shearer's Entropy Lemma}. For an exposition, see Jukna's textbook \cite[Thm 22.10]{Jukna2001}.
\begin{theorem} Let $A$ be a finite set and let $A_1, \ldots A_m\subseteq A$ be such that each $a\in A$ occurs in at least $k$ of the $A_i$. Let $\mathcal{S}\subseteq 2^A$. Then
$$|\mathcal{S}|^k\leq \prod_{i=1}^m |\mathcal{S}_i|,$$
where $\mathcal{S}_i:=\{S\cap A_i: S\in\mathcal S\}$ for $i=1,\ldots,m$. \noproof
\end{theorem}
We use the Product Theorem to bound the number of stable sets in induced subgraphs of which we only know the number of vertices.
\begin{lemma} Let $G=(V,E)$ be a vertex-transitive undirected graph and let $U\subseteq V$. Then $$\frac{\log i(G)}{|V|}\leq \frac{\log i(G[U])}{|U|}. $$
\end{lemma}
\proof Let $\Gamma$ be the automorphism group of $G$. Let $U^g:=\{g(u): u\in U\}$ for each $g\in \Gamma$. As $\Gamma$ acts transitively on $G$, we have $\left|\{g\in \Gamma: v\in U^g\}\right|= |\Gamma|\frac{|U|}{|V|}$ for each $v\in V$. Thus the sets $(U^g: g\in \Gamma)$ form a regular cover of $V$. Applying the Product  Theorem to bound the cardinality of $\mathcal{S}:=\{S\subseteq V: S\text{ a stable set of }G\}$, we find that 
$$|\mathcal{S}|^k\leq \prod_{g\in\Gamma} |\mathcal{S}_g|,$$
where $k= |\Gamma|\frac{|U|}{|V|}$ and $\mathcal{S}_g:=\{S\cap U^g: S\in\mathcal{S}\}$. Noting that 
$|\mathcal{S}|=i(G)$ and $|\mathcal{S}_g|=i(G[U^g])$, we derive that 
$$|\Gamma|\frac{|U|}{|V|} \log i(G)\leq   \sum_{g\in\Gamma}\log i(G[U^g])=|\Gamma| \log i(G[U]),$$
where the latter equation follows as each  $G[U^g]$ is isomorphic to $G[U]$. \endproof
If $G=(V,E)$ is a graph and  $U, U'\subseteq V$, then we denote $\delta(U,U'):=\{uv\in E: u\in U, v\in U'\}$.
\begin{lemma} Let $G=(V,E)$ be an undirected graph, and let $U, U'\subseteq V$ be disjoint sets so that $\delta(U,U')=\emptyset$. Then 
$\log i(G[U]) + \log i(G[U'])\leq \log i(G).$
\end{lemma}
\proof If $S\subseteq U$ and $S'\subseteq U'$ are both stable sets of $G$, then  $E(G[S\cup S'])=\delta(S,S')\subseteq \delta(U,U')=\emptyset$, so $S\cup S'$ is a stable set as well. Hence $i(G[U])i(G[U'])\leq i(G)$ and the lemma follows upon taking the logarithm.\endproof

\subsection{An upper bound on the fraction of nonbases in most matroids}
\begin{lemma} \label{lem:big_u}Let $u\in [0,1]$ and let $U\subseteq \binom{[n]}{r}$ be such that $|U|\geq  u\binom{n}{r}$. Then 
$$\log\left|\{M\in \MM_{n,r}: U(M)=U\}\right|\leq (1-u)\log s(n,r).$$
\end{lemma}
\proof $U$ is a set of vertices of the vertex-transitive graph $G:=J(n,r)$, and taking $U':=V(G)\setminus(U\cup N(U))$, we have $\delta(U,U')=\emptyset$.
If $M\in \MM_{n,r}$ and $U(M)=U$, then $W(M)$ is a stable set of $J(n,r)$ which is disjoint from $U\cup N(U)$, i.e. $W(M)\subseteq U'$. Hence
$$\log \left|\{W(M): M\in \MM_{n,r}, U(M)=U\}\right|\leq \log i(G[U'])\leq \log i(G) - \log i(G[U])\leq (1-u) \log i(G),$$
by applying the previous two lemmas. As $i(G)=i(J(n,r))=s(n,r)$, the lemma follows.\endproof

For a matroid $M$ of rank $r$ on $n$ elements, we denote $u(M):={|U(M)|}/{\binom{n}{r}}.$
We define $$\upsilon(n):=\frac{\Upsilon(n)+2\log (n+1)}{\log s(n)}.$$
As $\Upsilon(n)\leq O\left(\log(n)^3/n^2\binom{n}{n/2}\right)$ (Lemma \ref{lem:ns_ubound}) and $\log s(n)\geq \binom{n}{n/2}/ n$ (Theorem \ref{thm:knuth}), we have $\upsilon(n)\leq O\left( \frac{\log(n)^3}{n}\right)$.
\begin{theorem}\label{thm:small_u} Asymptotically almost all matroids on $n$ elements have $u(M)\leq \upsilon(n)$.
\end{theorem}
\proof Let $\mathcal{M}=\bigcup_{0\leq r\leq n} \mathcal{M}_{n,r}$, where
$\mathcal{M}_{n,r}:=\{M \in \mathcal{N}_{n,r}: u(M)> \upsilon(n)\}.$
Since almost all matroids are in $\mathcal{N}$, it suffices to show that $\mathcal{M}$ is thin. Clearly
$$|\mathcal{M}_{n,r}|\leq \sum_{U\in \mathcal{U}_{n,r}} \left|\{M\in \MM_{n,r}: U(M)=U, ~u(M)> \upsilon(n)\}\right|.$$
Since $\log |\mathcal{U}_{n,r}|\leq \Upsilon(n)$ for all $r$, and $\log |\{M\in \MM_{n,r}: U(M)=U, u(M)> \upsilon(n)\}|\leq (1-\upsilon(n)) \log s(n,r)$ for each~$U$ by Lemma~\ref{lem:big_u}, we have
$$\log |\mathcal{M}_{n,r}|\leq  \Upsilon(n) +  (1-\upsilon(n))\log s(n,r) \le (\upsilon(n) \log s(n) -2\log (n+1)) + (1-\upsilon(n))\log s(n).$$
Hence $\log |\mathcal{M}_{n,r}|\leq \log s(n)-2\log (n+1)$, or equivalently $|\mathcal{M}_{n,r}|\leq s(n)/(n+1)^2$ for each $r$, so that 
$$\frac{|\MM_{n}\cap\mathcal{M}|}{m(n)}\leq \frac{\sum_r|\mathcal{M}_{n,r}|}{s(n)}\leq \frac{(n+1) \max_r|\mathcal{M}_{n,r}|}{s(n)}\leq \frac{1}{n+1}\rightarrow 0$$
as~$n\rightarrow \infty$, as required.
\endproof
For a matroid $M$ on $n$ elements of rank $r$, we denote the fraction of dependent $r$-sets by $$d(M):=\frac{|U(M)\cup W(M)|}{\binom{n}{r}}.$$
\begin{theorem}\label{thm:small_d}Asymptotically almost all matroids on $n$ elements have $d(M)\leq \upsilon(n)+ 2/n$.
\end{theorem}
\proof Let $M\in \MM_{n,r}$. The dependent sets of $M$ with the cardinality of a base are $U(M)\cup W(M)$. The set $W(M)$ is a stable set of the Johnson graph $J(n,r)$, hence cannot contain more than a $2/n$ fraction of the $r$-sets. Hence $u(M)\geq d(M)-2/n$. The theorem follows from Theorem \ref{thm:small_u}.\endproof
As $\upsilon(n)\leq O(\log(n)^3/n)$ and $\upsilon(n)+ 2/n\leq O(\log(n)^3/n)$, Theorem \ref{thm:small_d} qualitatively puts the same bound on $d(M)$ as Theorem \ref{thm:small_u} puts on $u(M)$. Applying Theorem  \ref{thm:small_d}  has the benefit that we need not distinguish the circuit-hyperplanes from the other dependent sets. However, where we cannot expect to improve the asymptotic bound on $d(M)$ below $1/n$, the bound on $u(M)$ could still be significantly better, and  therefore we rely on the bound on $u(M)$ where we can in what follows.

\subsection{Uniform matroid minors}
\begin{lemma} Let~$0 \le a \le b$ and~$0 \le r \le n$ be integers satisfying~ $a \le r$ and $b-a \le n-r$. If $M \in \MM_{n,r}$ is a matroid without the uniform matroid $U_{a,b}$ as a minor, then $d(M)\geq 1/\binom{b}{a}$.\end{lemma}
\proof Let $C,D\subseteq E:=E(M)$ be such that $|C|=r-a$, $|D|=(n-r)-(b-a)$, and $C\cap D=\emptyset$. Consider
$$[C;D]:=\left\{X\in \binom{E}{r}: C\subseteq X, D\cap X=\emptyset\right\}.$$
If $C$ is dependent or $D$ is codependent, then all elements of $[C;D]$ are dependent. Otherwise, $M/C\backslash D$ is a minor of rank $a$ on $b$ elements $E\setminus(C\cup D)$, which by assumption cannot be uniform. Thus at least one set $Y\subseteq E\setminus(C\cup D)$ with $|Y|=a$ is not a basis of $M/C\backslash D$, and then $X=Y\cup C\in [C;D]$ is dependent. Summarising, there is at least one dependent set among the $\binom{b}{a}$ elements of $[C;D]$. Summing these lower bounds over all such $C, D$, we find that at least a $1/\binom{b}{a}$ fraction of the $X\in \binom{E}{r}$ is dependent.\endproof

\begin{theorem}\label{thm:uniform}
For all $c<1/2$, asymptotically almost all matroids on $n$ elements have a $U_{k, 2k}$-minor whenever $k\leq c\log n$.
\end{theorem}
\proof Asymptotically almost all matroids $M$ on $n$ elements have $d(M)\leq \upsilon(n)+2/n$. If $M$ does not have 
$U_{k,2k}$ as a minor, then $d(M)\geq 1/\binom{2k}{k}$. If both conditions on $d(M)$ hold, then 
$$2^{-2k}\leq 1/\binom{2k}{k}\leq d(M)\leq \upsilon(n)+2/n\leq O(\log(n)^3/n),$$
which fails if $k\leq c\log(n)$. The theorem follows.\endproof

\subsection{Matroid girth} The {\em girth} of a matroid $M$ is the smallest cardinality of a circuit $C$ of $M$.
\begin{lemma}\label{lem:girth} Let $M\in \MM_{n,r}$, and suppose that $M$ has a circuit of cardinality $k<r$. Then 
$u(M)\geq \left(\frac{r-k}{n}\right)^k.$\end{lemma}
\proof
If $M$ has a circuit $C$ of cardinality $k$, then any $X\in\binom{E}{r}$ with $C\subseteq X$ is dependent, and in fact $\left\{X\in \binom{E}{r}: C\subseteq X\right\}\subseteq U(M)$ as this set induces a connected subgraph of the Johnson graph. Hence
$$u(M)\geq \binom{n-k}{r-k}/\binom{n}{r}\geq \left(\frac{r-k}{n}\right)^k,$$
as required.\endproof

\begin{theorem}\label{thm:girth}For all $c < 1/2$, asymptotically almost all matroids on $n$ elements have girth at least $c\log(n)$.\end{theorem}
\proof Let  
$\mathcal{M}_{n}:=\{M\in \MM_{n}: u(M)\leq \upsilon(n), n-\sqrt{n}\leq r(M)\leq n+\sqrt{n}\}.$
By Theorem \ref{thm:rank_bound}, the rank $r$ of almost all matroids on $n$ elements satisfies $|r-n/2|\leq \sqrt{n}$, and by Theorem \ref{thm:small_u}, almost all matroids $M\in \mathcal{M}_n$ have $u(M)\leq \upsilon(n)$. Hence almost all matroids are in $\mathcal{M}:=\bigcup_n \mathcal{M}_n$.

Let $M\in \mathcal{M}_{n}$, and suppose that $M$ has a circuit of cardinality $k$, where~$k\leq n/12$. As $ n-\sqrt{n}\leq r(M)\leq n+\sqrt{n}$, we have $n/3 \leq r\leq 2n/3$ for sufficiently large $n$ and hence $u(M)\geq \left(\frac{r-k}{n}\right)^k\geq 4^{-k}$ by Lemma~\ref{lem:girth}.
On the other hand we have $u(M)\leq \upsilon(n)$, and hence
$$4^{-k}\leq u(M)\leq O(\log(n)^3/n)\text{ as }n\rightarrow \infty,$$
which fails if~$k \le c\log(n)$.

\subsection{Matroid connectivity} Let $M$ be a matroid on $E$. A   partition $\{A,B\}$ of $E$ is {\em $k$-separating} if 
$$r_M(A)+r_M(B)<r(M)+k$$
and $\{A,B\}$ is a {\em $k$-separation} if in addition $|A|,|B|\geq k$. The {\em connectivity} of $M$ is the smallest $k$ such that $M$ has a $k$-separation.   

\begin{lemma}\label{lem:sep} Let $c'>0$. There is a constant $c>0$ so that for sufficiently large $n$ and any $r$ such that $|r-n/2|\leq \sqrt{n}$, we have: if $M\in \MM_{n,r}$  has girth and cogirth at least $c'\log(n)$, and $M$ has a  $k$-separation $\{A,B\}$, then $$u(M)\geq 1-kc/\sqrt{\log(n)}.$$
\end{lemma}
\proof Let $M\in \MM_{n,r}$  have girth and cogirth at least $c'\log(n)$, and let $\{A,B\}$ be a  $k$-separation of $M$. Consider the sets 
$\tilde{U}_A:=\left\{X\in \binom{E}{r}: |X\cap A|>r_M(A)\right\}\text{ and }\tilde{U}_B:=\left\{X\in \binom{E}{r}: |X\cap B|>r_M(B)\right\}.$
Each $X\in \tilde{U}_A\cup \tilde{U}_B$ is dependent, and both $\tilde{U}_A$ and $\tilde{U}_B$ induce connected subgraphs of the Johnson graph $J(n,r)$. It follows that $\tilde{U}_A\cup \tilde{U}_B\subseteq U(M)$. Writing $a:=|A|$ and $b:=|B|$, the number of $r$-sets not in $\tilde{U}_A\cup \tilde{U}_B$ is 
$$q:=\sum\left\{\binom{a}{s}\binom{b}{t}: ~s\leq r_M(A), ~t\leq r_M(B), ~s+t=r\right\}.$$
As $r_M(A)+r_M(B)<r(M)+k$, there are at most $k$ terms in the sum, hence
$$q\leq k\max_{s,t}\binom{a}{s}\binom{b}{t}\leq k \left(\sqrt{\frac{2}{\pi}}\frac{2^a}{\sqrt{a}}\right)\left(\sqrt{\frac{2}{\pi}}\frac{2^b}{\sqrt{b}}\right)=k\frac{2}{\pi}\frac{2^n}{\sqrt{ab}}\leq 2 k \left(\sqrt{\frac{2}{\pi}}\sqrt{\frac{n}{ab}}\right)\binom{n}{n/2}.$$
for sufficiently large $n$, using both sides of the asymptotic estimate $\sqrt{\frac{2}{\pi}}\frac{2^n}{\sqrt{n}}(1+o(1)) \leq \binom{n}{n/2}\leq \sqrt{\frac{2}{\pi}}\frac{2^n}{\sqrt{n}}$. There is a constant $c''$ so that  for all sufficiently large $n$ we have  $|r-n/2|\leq \sqrt{n} \Rightarrow \binom{n}{n/2}/\binom{n}{r}\leq c''$. Hence
$$q/\binom{n}{r}\leq 2 k \left(\sqrt{\frac{2}{\pi}}\sqrt{\frac{n}{ab}}\right)\binom{n}{n/2}/\binom{n}{r}\leq 2c'' k \left(\sqrt{\frac{2}{\pi}}\sqrt{\frac{n}{ab}}\right).$$
Since $\{A,B\}$ is a $k$-separation, both $A$ and $B$ contain either a circuit or a cocircuit of $M$. By our assumption on the girth and the cogirth of $M$, it follows that $a,b\geq c'\log(n)$. Maximizing the upper bound on $q$ over $a,b\geq c'\log(n)$ such that $a+b=n$, we find that the upper bound is attained by $a=c'\log(n), b=n-a$. Hence there is a constant $c$ such that for sufficiently large $n$ we have $q/\binom{n}{r}\leq k c/\sqrt{\log(n)}$. We can now bound
$$u(M)=\frac{|U(M)|}{\binom{n}{r}}\geq \frac{|\tilde{U}_A\cup\tilde{U}_B|}{\binom{n}{r}}\geq 1-\frac{q}{\binom{n}{r}}\geq 1- k c/\sqrt{\log(n)}$$
as required.
\endproof

\begin{theorem} There is a $c>0$ so that asymptotically almost all matroids on $n$ elements have connectivity at least $c\sqrt{\log(n)}$.\end{theorem}
\proof 
Consider the class of matroids
$$\mathcal{M}^c_{n}:=\left\{M\in \MM_{n,r}:  M\text{ has a $k$-separation } (A,B), \text{ with }k\leq c\sqrt{\log n}\right\}$$
To prove the theorem it will suffice to show that $\mathcal{M}^c=\bigcup_n \mathcal{M}^c_{n}$ is thin for some $c>0$.

By Theorem \ref{thm:girth} there is a $c'>0$ so that almost all matroids on $n$ elements have have girth and cogirth at least $c'\log(n)$, and  by Theorem \ref{thm:rank_bound}, almost all matroids on $n$ elements have rank $r$ between $n/2-\sqrt{n}$ and $n/2+\sqrt{n}$. Hence it will even suffice to show that $\mathcal{N}^c=\bigcup_n \mathcal{N}^c_{n}$ is thin, where 
$$\mathcal{N}^c_{n}:=\left\{M\in \mathcal{M}^c_{n}:  M\text{ has girth, cogirth } \geq c'\log n, \text{ and }|r(M)-n/2|\leq \sqrt{n}\right\}.$$
By Lemma \ref{lem:sep}, there is a constant $c''>0$ so that $u(M)\geq 1- kc''/\sqrt{\log(n)}\geq 1 - cc''$ for all $M\in \mathcal{N}^c_{n}$, provided that $n$ is sufficiently large. Pick any $c>0$ such that $cc''<1$. Then 
$u(M)\geq 1- cc''>0$ for all $M\in \mathcal{N}^c_{n}$ provided that $n$ is large enough. The theorem then follows by an application of Theorem \ref{thm:small_u}. \endproof



\section{\label{sec:many}Most matroids do have some non-bases}
\subsection{Small stable sets in regular graphs}Let $i(G,m):=|\{S\subseteq V(G): S\text{ a stable set of }G, |S|\leq m\}|$.
\begin{theorem}\label{thm:stable2}Let $G=(V,E)$ be regular of degree $d\geq 1$, with smallest eigenvalue $-\lambda$. Then 
 $$ i(G,m)\leq \sum_{i=0}^{\left\lceil \sigma N\right\rceil} \binom{N}{i} \sum_{j=0}^{\alpha N} \binom{\alpha N}{j},$$
where $\sigma:=(\ln(d+1) + 1)/(d+\lambda)$, $\alpha:=\lambda/(d+\lambda)$, and $N:=|V|$.
\end{theorem}
\proof Let $K\subseteq V$ be a stable set of $G$ such that $|K|\leq m$. By Lemma \ref{lem:KW2}, there are sets $S, A\subseteq V$ with $|S|\leq \left\lceil\sigma N\right\rceil$, $|A|\leq \alpha N$, such that $S$ determines $A$ and $S\subseteq K\subseteq S\cup N(S)\cup A$. Since $K$ is a stable set $K\cap N(S)=\emptyset$, and hence $K=S\cup (K\cap A)$, where $|K\cap A|\leq |K|\leq m$. Thus the number of such $K$ is bounded by the number of subsets $S$ of size at most $\sigma N$ from a set of size $N$ times the number of subsets $K\cap A$ of size at most $m$ from a set of size $\alpha N$. So we have
$$i(G,m)\leq
\left(\sum_{i=0}^{\left\lceil\sigma N\right\rceil} \binom{N}{i}\right) \left(\sum_{j=0}^{m} \binom{\alpha N}{j}\right),
$$
as required.\endproof
\subsection{A lower bound on the fraction of nonbases in most matroids}
For a matroid $M$ of rank $r$ on $n$ elements $E$, we denote 
$$w(M):=\frac{|W(M)|}{\binom{n}{r}}.$$ Let $\digamma(n):=\frac{1}{5n}$.
\begin{lemma} \label{lem:small_stable} There is an $\varepsilon > 0$ such that for sufficiently large~$n$ we have
 $$\log i\left(J(n,r),\digamma(n)\binom{n}{r}\right)\leq (1-\epsilon)\log s(n)$$
 for all $0 < r < n$.
\end{lemma}
\proof Fix $n, r$, and put $N:=\binom{n}{r}$. By Theorem~\ref{thm:stable2}, we have
$$\log i(J(n,r),\digamma(n)N)\leq \log \sum_{i=0}^{\left\lceil \sigma_{n,r} N\right\rceil} \binom{N}{i} + \log \sum_{j=0}^{\digamma(n) N} \binom{\alpha_{n,r} N}{j}
$$
Using~$N \le \binom{n}{n/2}$, and the bounds on~$\left\lceil \sigma_{n,r} N\right\rceil$ and $\alpha_{n,r} N$ from Lemma~\ref{lem:sigma}, we obtain
$$
\log \sum_{i=0}^{\left\lceil \sigma_{n,r} N\right\rceil} \binom{N}{i} \leq \frac{9\ln n}{n^2} \binom{n}{n/2} \log \left(\frac{\e n^2}{9 \ln n}\right)
\quad\text{and}\quad
\log \sum_{j=0}^{\digamma(n) N} \binom{\alpha_{n,r} N}{j} \leq \frac{\log(10\e)}{5n} \binom{n}{n/2}
$$
As ${\log(10\e)}/{5}<0.96$ and $\binom{n}{n/2}/n\leq s(n)$, this proves the lemma.
\endproof
\begin{theorem}\label{thm:big_w} Asymptotically almost all matroids on $n$ elements have $w(M)\geq \digamma(n)$.\end{theorem} 
\proof Let $\mathcal{M}_{n,r}:=\{M\in \NS_{n,r}: w(M)< \digamma(n)\}$. It suffices to show that $\mathcal{M}:=\bigcup \mathcal{M}_{n,r}$ is thin. Each matroid $M\in \mathcal{M}_{n,r}$ is determined by the pair $U(M), W(M)$, where $U(M)\in \mathcal{U}_{n,r}$ and $W(M)$ is a stable set of the Johnson graph $J(n,r)$ with $|W(M)|\leq m:=\digamma(n)\binom{n}{r}$. Hence
$$|\mathcal{M}_{n,r}|\leq |\mathcal{U}_{n,r}| i\left(J(n,r),m\right). $$
Using Lemma \ref{lem:small_stable}, $\log i(J(n,r),m)\leq (1-\epsilon)\log s(n)$ for sufficiently large $n$ and by Lemma \ref{lem:ns_ubound} we have $\log |\mathcal{U}_{n,r}|\leq \Upsilon(n)\binom{n}{n/2}\leq o(\log s(n))$. Hence
$$\log\frac{ |\mathcal{M}_{n,r}|}{s(n)}\leq o(\log s(n))+(1-\epsilon)\log s(n)-\log s(n)=(-\epsilon+o(1))\log s(n)<-2\log(n+1)$$
for sufficiently large $n$, so that $$\frac{|\mathcal{M}\cap\MM_{n}|}{m(n)}\leq \frac{\sum_r|\mathcal{M}_{n,r}|}{s(n)}\leq (n+1)\frac{\max_r|\mathcal{M}_{n,r}|}{s(n)}\leq \frac{1}{n+1}\rightarrow 0$$
when $n\rightarrow \infty$, as required.\endproof

\subsection{Truncations of matroids} If  $M$ is a matroid of rank $r$ on ground set $E$, the {\em truncation} $T(M)$ is the matroid on the same ground set $E$ whose independent sets are 
$$\{X\subseteq E: |X|\leq r-1, X\text{ independent in }M\}.$$
\begin{theorem} Asymptotically almost no matroids arise as the truncation of another matroid.\end{theorem}
\proof Let $\mathcal{M}_n:=\{T(M): M\in \MM_n\}$. If $M'$ arises from $M$ by relaxing a circuit-hyperplane $X\in W(M)$, then $T(M')=T(M)$ as $r_M(X)\geq r(M)-1$. Thus 
$$\mathcal{M}_n=\{T(M): M\in \MM_n, W(M)=\emptyset\}.$$
It follows that for each $n$,  $\mathcal{M}_n$ has at most as many members as  $\{M\in \MM_n: W(M)=\emptyset\}$. By Theorem \ref{thm:big_w}, $\{M\in \MM: W(M)=\emptyset\}$ is thin, and hence $\mathcal{M}=\bigcup_n \mathcal{M}_n$ is thin.
 \endproof




\section{\label{sec:discussion}Final remarks}
\subsection{A conjecture} We think that the following strengthening of Theorem \ref{thm:ns} is worth aiming for.
\begin{conjecture} \label{conj:ns}There is an $\epsilon>0$ and a class of matroids $\NS\subseteq \MM$ containing almost all matroids, such that 
$$\log \left|\{U(M): M\in  \MM_{n,r}\cap \NS\}\right|\leq 2^{-\epsilon n}\binom{n}{n/2}$$
for all $0\leq r\leq n$.
\end{conjecture} 
This would translate to an exponential upper bound of $\upsilon'(n)\sim 2^{-\epsilon n}$ on the density $u(M)$ of almost all matroids, which would in turn give a linear lower bound on the girth of $\sim \log (1/\upsilon'(n))=\Omega(n)$,  for almost all matroids. 

\ignore{In the present paper, the bound on $\mathcal{U}_{n,r}=\left|\{U(M): M\in  \MM_{n,r}\cap \NS\}\right|$ is derived from a compressed description of $U(M)$, the vertices in `big components' of $J(n,r)\setminus\BB(M)$.
Although our encoding scheme yields a compressed description of $U(M)$ as a whole, there seems to be no a priori reason why the bounds on the length of a description of $U(M)$ could not exist on the level of individual components of $J(n,r)\setminus \BB$. We find it likely that the amount of information needed to describe a component $C$ is $o(|C|/(\log s(n,r)))$, as $|C|\rightarrow \infty$. In that case, matroids such that $U(M)$ contains `very big' components must be rare by an application of Lemma \ref{lem:big_u}.

}

\subsection{Bounding other classes of matroids}
We applied Theorem \ref{thm:ns} essentially through its Corollary~\ref{cor:thin} to show that various classes $\mathcal{M}$ of matroids are thin. Our main tools for bounding the number of distinct sets of circuit-hyperplanes $W(M)$ which arise from matroids $M\in \mathcal{M}$ were combinatorial, based on a lower bound on the cardinality of  $U(M)$ or a uniform upper bound on the cardinality of $W(M)$ itself. 

We doubt that this exhausts the possibilities. To illustrate, we cite two conjectures from \cite{MayhewNewmanWelshWhittle2011}.
\begin{conjecture}Asymptotically almost every matroid is not representable over any field.
\end{conjecture}
This conjecture was recently proved by Nelson~\cite{Nelson2016}. It seems amenable to an argument based on Theorem \ref{thm:ns} as well. Relaxing  a circuit-hyperplane in a linear matroid often results in a non-linear matroid, a phenomenon which has been exploited to prove that any real-representable matroid is the minor of an excluded minor for real-representability \cite{MayhewNewmanWhittle2009}. Geelen, Gerards and Whittle note in \cite{GGW2014} that ``while the operation of relaxing a circuit-hyperplane does not behave well with respect to representation in general, it behaves particularly poorly with respect to representation over finite fields.'' Such limited flexibility may very well translate to bounds 
which will complement Corollary \ref{cor:thin}.
\begin{conjecture}\label{conj:sparseminor}Let $N$ be a fixed sparse paving matroid. Asymptotically almost every matroid has an $N$-minor.
\end{conjecture}
The case of this conjecture where $N$ is any uniform matroid is settled by Theorem \ref{thm:uniform}.  As far as we know, the conjecture is open even for tiny matroids such as $N=W^3$ or $N=M(K_4)$. This does not bode well for the difficulty of the general problem. Even if almost all matroids are sparse paving, we are left with the following  hard variant of Conjecture \ref{conj:sparseminor}.
\begin{conjecture}Let $N$ be a fixed sparse paving matroid. Asymptotically almost every sparse paving matroid has an $N$-minor.\end{conjecture}
Recently, Critchlow~\cite{Critchlow2016} proved this conjecture for several sparse paving matroids.
The underlying question seems to be whether the large stable sets of the Johnson graphs collectively have `structure'. If so, most stable sets may avoid certain sub-configurations and most sparse paving matroids will avoid having certain sparse paving minors. If not, the incidence of a specific minor should be according to largely independent probabilities of seeing circuit-hyperplanes in a minor, upon contracting and deleting random sets of an appropriate size. In that case, the conjecture should hold.

\ignore{
\subsection{Excluding a uniform minor}
In \cite{Geelen2008}, Geelen propose that matroids without a fixed uniform minor have a particular structure which would necessarily put tight limits on the number of such matroids. Theorem \ref{thm:uniform} implies that any minor-closed class of matroids which does not contain all uniform matroids is necessarily thin, and in this sense lends support to these conjectures. 

We have shown that almost all matroids are $k$-connected, but in the present context the following would be more useful.
\begin{conjecture}\label{conj:uniform}Let $k,l\in \N$. Asymptotically almost all matroids without a $U_{k,2k}$-minor are either vertically $l$-connected or cyclically $l$-connected.\end{conjecture}
}

\subsection{Entropy counting}
We counted matroids by conditioning on the value of $U(M)$. The statement of the main result in this paper can be cast in the language of entropy. For let $M$ be a random matroid that is drawn uniformly from a class of matroids $\mathcal{M}\subseteq \MM_{n}$. Then Theorem \ref{thm:ns} implies that the entropy of the derived random variable $U(M)$ is bounded by $$H(U(M))\leq c\log(n)^3/n^2\binom{n}{n/2}$$
so that bounding $\log |\mathcal{M}|=H(M)=  H(U(M))+ H(M|U(M))$ strictly below $\log m(n)\geq \binom{n}{n/2}/n$ amounts to sufficiently bounding the conditional entropy $H(M|U(M))$. In at least one setting, the case where $U(M)$ is assumed to be large for all $M\in \mathcal{M}$, this scheme works through an argument which itself relies on entropy reasoning, namely Shearer's Lemma. Theorem \ref{thm:ns} itself however follows from an encoding argument of which we could not see any `entropic' version. Since we hope that pulling the entire argument within the realm of entropy counting will yield further insights and shortcuts, we ask if the Kleitman-Winston encoding procedure can somehow be matched by an entropy counting argument.

\section{Acknowledgement}
The first author thanks Nathan Bowler, Jim Geelen, Peter Nelson, and Pascal Schweitzer for a lively discussion about an early version of this material, during the Graph Theory meeting in Oberwolfach in January 2016. Peter Nelson pointed out on this occasion that the $k$-connectivity of almost all matroids should also follow from the asymptotic sparsity of non-bases.

We would like to thank the referees, whose detailed comments allowed us to improve the exposition in several places.

\bibliographystyle{alpha}
\bibliography{bib}

\end{document}